\documentclass[12pt,twoside]{article}
\usepackage{amssymb,amsmath,amsthm}
\usepackage[noadjust]{cite}

\usepackage[german,english]{babel}
\usepackage{geometry,xcolor}
\DeclareSymbolFont{bbold}{U}{bbold}{m}{n}
\DeclareSymbolFontAlphabet{\mathbbold}{bbold}
\newcommand{\ind}{\mathbbold{1}}

\newcommand{\C}{\mathbb{C}}

\newcommand{\N}{\mathbb{N}}
\newcommand{\R}{\mathbb{R}}

\newcommand{\A}{\textup{(A)}}
\newcommand{\Ls}{\mathcal{L}}
\newcommand{\tT}{\widetilde T}
\newcommand{\tH}{\widetilde H}
\newcommand{\tomega}{\widetilde\omega}
\newcommand{\lB}{\mkern-1mu B \mkern1mu}
\newcommand{\Proof}{\hskip-0.05em Proof }

\renewcommand{\Re}{\operatorname{Re}}

\newcommand{\from}{\colon}

\newcommand{\rfrac}[2]{\tfrac{#1}{\raisebox{0.1em}{\scriptsize$#2$}}}

\newcommand{\tup}[1]{\textup{(#1)}}

\renewcommand\le{\leqslant}
\renewcommand\ge{\geqslant}
\renewcommand{\|}{|\!|}

\newcommand{\set}[2]{\bigl\{#1\setcol#2\bigr\}}
\newcommand{\sset}[2]{\{#1\setcol#2\}}
\newcommand{\setcol}{\nobreak\mskip2mu\mathclose{}
  \mathopen{;}\penalty300\mskip6mu plus3mu minus1mu\relax}

\newcommand{\eps}{\varepsilon}

\newcommand\operator[1]{\expandafter\newcommand\csname#1\endcsname{\operatorname{#1}}}
\operator{vol}
\operator{spt}
\operator{dist}

\swapnumbers
\newtheorem{theorem}{Theorem}[section]

\newtheorem{lemma}[theorem]{Lemma}
\newtheorem{proposition}[theorem]{Proposition}

\theoremstyle{definition}

\newtheorem{remark}[theorem]{Remark}

\numberwithin{equation}{section}

\newcommand{\avoidbreak}{\postdisplaypenalty50}

\makeatletter

\let\save@text@hyphen\-
\renewcommand*\-{\ifmmode\mkern-1.5mu\else\save@text@hyphen\fi}
\newcommand{\+}{\mkern1mu}

\renewcommand{\.}{.\nobreak\@ifnextchar1{\hskip0.18em plus0.05em minus
0.02em}{\hskip0.22em plus0.06em minus 0.04em}}
\renewcommand\section{\@startsection {section}{1}{\z@}%
                                     {-3.25ex \@plus -1ex \@minus -.2ex}%
                                     {1.5ex \@plus.2ex}%
                                     {\normalfont\large\bfseries}}
\def\env@cases{%
  \let\@ifnextchar\new@ifnextchar
  \left\lbrace
  \def\arraystretch{1.1}%
  \array{@{\,}l@{\quad}l@{}}%
}
\@ifundefined{MoveEqLeft}{%
\newcommand\MoveEqLeft[1][2]{%
  \global\@tempdima=#1em%
  \kern\@tempdima%
  &
  \kern-\@tempdima}
}{}

\newcommand*\if@single[3]{%
  \setbox0\hbox{${\mathaccent"0362{#1}}^H$}%
  \setbox2\hbox{${\mathaccent"0362{\kern0pt#1}}^H$}%
  \ifdim\ht0=\ht2 #3\else #2\fi
  }
\newcommand*\rel@kern[1]{\kern#1\dimexpr\macc@kerna}
\newcommand*\widebar[1]{\@ifnextchar^{{\wide@bar{#1}{0}}}{\wide@bar{#1}{1}}}
\newcommand*\wide@bar[2]{\if@single{#1}{\wide@bar@{#1}{#2}{1}}{\wide@bar@{#1}{#2}{2}}}
\newcommand*\wide@bar@[3]{%
  \begingroup
  \def\mathaccent##1##2{%
    \if#32 \let\macc@nucleus\first@char \fi
    \setbox\z@\hbox{$\macc@style{\macc@nucleus}_{}$}%
    \setbox\tw@\hbox{$\macc@style{\macc@nucleus}{}_{}$}%
    \dimen@\wd\tw@
    \advance\dimen@-\wd\z@
    \divide\dimen@ 3
    \@tempdima\wd\tw@
    \advance\@tempdima-\scriptspace
    \divide\@tempdima 10
    \advance\dimen@-\@tempdima
    \ifdim\dimen@>\z@ \dimen@0pt\fi
    \rel@kern{0.6}\kern-\dimen@
    \if#31
      \overline{\rel@kern{-0.6}\kern\dimen@\macc@nucleus
                \rel@kern{0.4}\kern\dimen@}%
      \advance\dimen@0.4\dimexpr\macc@kerna
      \let\final@kern#2%
      \ifdim\dimen@<\z@ \let\final@kern1\fi
      \if\final@kern1 \kern-\dimen@\fi
    \else
      \overline{\rel@kern{-0.6}\kern\dimen@#1}%
    \fi
  }%
  \macc@depth\@ne
  \let\math@bgroup\@empty \let\math@egroup\macc@set@skewchar
  \mathsurround\z@ \frozen@everymath{\mathgroup\macc@group\relax}%
  \macc@set@skewchar\relax
  \let\mathaccentV\macc@nested@a
  \if#31
    \macc@nested@a\relax111{#1}%
  \else
    \def\gobble@till@marker##1\endmarker{}%
    \futurelet\first@char\gobble@till@marker#1\endmarker
    \ifcat\noexpand\first@char A\else
      \def\first@char{}%
    \fi
    \macc@nested@a\relax111{\first@char}%
  \fi
  \endgroup
}
\makeatother

\begin{document}

\medmuskip=4mu plus 2mu minus 3mu
\thickmuskip=5mu plus 3mu minus 1mu
\belowdisplayshortskip=\belowdisplayskip
\newcommand\smallbds{\vskip-1\lastskip\vskip5pt plus3pt minus2pt\noindent}

\title{\Large $L_\infty$-estimates for the torsion function and \\
$L_\infty$-growth of semigroups satisfying Gaussian bounds}
\author{Hendrik Vogt%
\footnote{Fachbereich 3 -- Mathematik, Universit{\"a}t Bremen, 28359 Bremen, Germany,
+49\,421\,218-63702,
{\tt hendrik.vo\rlap{\textcolor{white}{hugo@egon}}gt@uni-\rlap{\textcolor{white}{hannover}}bremen.de}}}
\date{}

\maketitle

\begin{abstract}
We investigate selfadjoint $C_0$-semi\-groups on Euclidean domains satisfying
Gaussian upper bounds. Major examples are semigroups generated by second order
uniformly elliptic operators with Kato potentials and magnetic fields.
We study the long time behaviour of the $L_\infty$ operator norm of the semigroup.
As an application we prove a new $L_\infty$-bound for the torsion function
of a Euclidean domain that is close to optimal.

\vspace{8pt}


\vspace{2pt}

\phantom{Keywords: }%
$L_\infty$-estimates
\end{abstract}

\section{Introduction and main results}\label{sec_intro}

If $H$ is a Schr\"odinger operator in a Euclidean domain with ground state energy
$\inf\sigma(H) = 0$, then clearly the $L_2$ operator norm of $e^{-tH}$ equals $1$
for all $t\ge0$.
By duality and interpolation, the $L_\infty$ operator norm is bounded below by~$1$,
and typically one has an upper bound that grows polynomially in~$t$,
with an exponent that depends on the dimension.
This phenomenon was studied, e.g., in \cite{sim80}, \cite{dasi91}, \cite{ouh06a};
see the discussion below for more details.
In Theorem~\ref{main-sg} we give an improvement of the growth bound shown
by Ouhabaz in \cite{ouh06a}.

If $H=-\Delta_D$ is the negative Dirichlet Laplacian on a bounded Euclidean domain~$D$,
then $H$ is invertible,
and one can compare the $L_2$ and $L_\infty$ operator norms of~$H^{-1}$.
This leads to $L_\infty$-estimates for the torsion function $u_D = H^{-1}\ind$ of~$D$;
cf.\ Remark~\ref{torsion-rem}.
Following the work of van den Berg and Carroll \cite{beca09}, we prove an estimate
for $\|u_D\|_\infty$ in terms of the ground state energy $\inf\sigma(-\Delta_D)$,
with a dimension dependent constant that is close to optimal.
This is obtained as a consequence of the more general Theorem~\ref{torsion}.

\pagebreak[1]\smallskip

The following are our standing assumptions.
\begin{itemize}
\item[\A]
Let $\Omega\subseteq\R^d$ be measurable, where $d\in\N$,
and let $H$ be a selfadjoint operator in $L_2(\Omega)$ with the following properties:
\[
  E_0 := E_0(H) := \inf \sigma(H) > -\infty,
\]
and the $C_0$-semi\-group generated by $-H$ satisfies Gaussian upper bounds:
$e^{-tH}$ has an integral kernel $p_t$, for every $t>0$,
and there exist $M\ge1$,\, $\omega\in\R$ and $a>0$ such that
\begin{equation}\label{gub}
  |p_t(x,y)| \le Me^{\omega t} \cdot (a\pi t)^{-d/2} 
                 \smash{ \exp\left(\! -\frac{|x-y|^2}{at} \right) }
  \avoidbreak
\end{equation}
for all $t>0$ and a.e.\ $x,y\in\Omega$.
\end{itemize}
Expressed differently, \eqref{gub} means that $e^{-tH}$ is dominated by
$Me^{\omega t} e^{\frac{a}{4}t\Delta}$,
\begin{equation}\label{gub2}
  |e^{-tH}\-f| \le Me^{\omega t} e^{\frac{a}{4}t\Delta}|f| \qquad
  \bigl(t\ge0,\ f \in L_2(\Omega)\bigr),
  \avoidbreak
\end{equation}
where $\Delta=\Delta_{\R^d}$ denotes the Laplacian on $\R^d$
and $|f|$ is extended by $0$ to a function on all of $\R^d$.

Assumption~\A\ is satisfied, e.g., if $\Omega$ is an open set and
$H$ is a selfadjoint second order uniformly elliptic operator in divergence form
subject to Dirichlet boundary conditions; see \cite[Thm\.3.2.7]{dav89}.
If some regularity of the boundary of $\Omega$ is assumed,
then a variety of other boundary conditions are covered, such as
Neumann or Robin boundary conditions (\cite[Thm\.6.10, Prop\.4.24]{ouh05}).
Moreover, $H$ may include a magnetic field (due to the diamagnetic inequality)
and a potential from the Kato class; cf.\ \cite[Prop\.B.6.7]{sim82}.

The first topic of the paper is to investigate the asymptotic behaviour of
$\|e^{-tH}\|_{\infty\to\infty}$, i.e., of the norm of the semigroup operators in
$\Ls(L_\infty(\Omega))$.
In \cite[formula~(1.9)]{sim80}, the estimate
\[
  \|e^{-tH}\|_{\infty\to\infty} \le C(1+t)^{d/2} e^{-E_0t}
\]
was shown for a certain class of Schr\"odinger operators $H$ on $\R^d$.
This estimate was substantially improved and generalized in \cite[Thm\.4]{ouh06a}
for operators on Euclidean domains and in \cite[Thm\.7]{ouh06b} for operators on
open subsets of complete Riemannian manifolds:
If the semigroup generated by $-H$ satisfies Gaussian upper bounds, then
\[
  \|e^{-tH}\|_{\infty\to\infty} \le C(1+t\ln t)^{d/4} e^{-E_0t}.
\]

Here we show that the term $\ln t$ can be removed, albeit only for the case of operators
on subsets of $\R^d$. We point out that all the constants in our estimate are explicit.

\begin{theorem}\label{main-sg}
Let Assumption~\A\ hold. Then
\[
  \|e^{-tH}\|_{\infty\to\infty} \le 2^{1/4}M \bigl(1+\tfrac{5.56}{d}(E_0+\omega)t\bigr)^{d/4} e^{-E_0t}
  \avoidbreak
\]
for all $t\ge0$.
\end{theorem}

\begin{remark}
As pointed out in \cite{ouh06a}, the exponent $d/4$ is sharp in dimension $d=4$:
in \cite[Thm\.3.1]{sim81}, examples are given where $E_0=0$ and
$\|e^{-tH}\|_{\infty\to\infty}$ grows like $\frac{(1+t)^{\smash{4/4}}}{\ln t}$.
It is not clear if $d/4$ is sharp in other dimensions;
however, a reduction of the exponent below $(d-2)/4$ is impossible,
which can be seen by considering slowly varying resonances
with index $\alpha\in\bigl(0,(d-2)/2\bigr)$ (see \cite[Prop\.10, Thm\.14]{dasi91}).
\end{remark}

For our second main result we specialize to the case that Assumption~\A\ is satisfied
with $M=1$,\, $\omega=0$ and $a=4$, i.e., $e^{-tH}$ is dominated by the free heat
semigroup on $\R^d$,
\begin{equation}\label{free}
  |e^{-tH}\-f| \le e^{t\Delta}|f| \qquad \bigl(t\ge0,\ f \in L_2(\Omega)\bigr).
  \avoidbreak
\end{equation}
An important example for the operator $-H$ is the Dirichlet Laplacian with magnetic field and
a locally integrable absorption potential on an open set $\Omega\subseteq\R^d$
(cf.\ \cite[Section~4.5]{ouh05}).
For more general absorption potentials the space of strong continuity of
the semigroup will be $L_2(\Omega')$ for some measurable $\Omega'\subseteq\Omega$.

Assuming $E_0(H)>0$, we are going to study the quantity
\begin{equation}\label{q-def}
  q(H) := \frac{\|H^{-1}\|_{\infty\to\infty}}{\|H^{-1}\|_{2\to2}}
        = E_0(H) \cdot \|H^{-1}\|_{\infty\to\infty}\,.
\end{equation}
Note that $q(H) \ge 1$ by duality and Riesz-Thorin interpolation,
and $q(H) = 1$ if $H = -\Delta_{\R^d}+c$ for some $c>0$.
Also in the case where $H$ is the negative Dirichlet Laplacian on a bounded domain,
$q(H)$ may be arbitrarily close to~$1$, as was recently shown in \cite{ber17}.

\begin{remark}\label{torsion-rem}
(a) If $H^{-1}$ is a positivity preserving operator, then
\[
  \|H^{-1}\|_{\infty\to\infty} = \|H^{-1}\ind\|_\infty\,
\]\smallbds
so that
\begin{equation}\label{q-formula}
  q(H) = E_0(H) \cdot \|H^{-1}\ind\|_\infty
  \avoidbreak
\end{equation}\smallbds
in this case.

(b) Let $D$ be an open subset of $\R^d$ and $H = -\Delta_D$, where $\Delta_D$
denotes the Dirichlet Laplacian on $D$.
Then $u_D := H^{-1}\ind$ is the \emph{torsion function} of $D$,
and by~\eqref{q-formula} we have
\[
  \|u_D\|_\infty = \frac{q(H)}{E_0(H)}\,.
\]
Thus, the bounds for the quantity $q(H)$ that we prove in Theorem~\ref{torsion} below
lead to bounds for the $L_\infty$-norm of the torsion function~$u_D$.
We mention that
\[
  u_D(x) = \mathbb{E}_x \bigl( \inf\sset{t\ge0}{B(t)\notin D} \bigr),
\]
the expected time of first exit from $D$ of Brownian motion $B$ starting at $x\in D$.

\end{remark}

We first consider the case that $H=-\Delta_{B_d}$, where
$B_d$ denotes the open unit ball in~$\R^d$.

\begin{lemma}\label{ball}
There exists $C>0$ such that
\[
  \smash[b]{ \frac d8 \le q(-\Delta_{B_d}) \le \frac d8 + Cd^{1/3} }
\]
for all dimensions $d$.
\end{lemma}

Is is easy to see that $q(-\Delta_B) = q(-\Delta_{B_d})$ for all balls $B\subset\R^d$.
One might conjecture that $q(-\Delta_D) \le q(-\Delta_{B_d})$
for all bounded open subsets $D\subset\R^d$ (\cite[top of p\.614]{ber12}).
This does \emph{not} hold in dimension $d=2$: if $D$ is an equilateral triangle,
then $q(-\Delta_D) \approx 1.462 > 1.446 \approx q(-\Delta_{B_2})$;
see \cite[top of p\.116]{her79}.
It was recently shown by Henrot et al.\ (\cite[Cor\.3.7]{hlp17}) that
also the equilateral triangle is not a maximizer of $q(-\Delta_D)$ in dimension~$2$.

Our second main result shows that, nevertheless, $q(-\Delta_{B_d})$
is not far from an upper bound for $q(H)$,
even for the general operator $H$ in place of $-\Delta_D$;
note that our estimate below with $c\sqrt d$ is only slightly worse than
the estimate with $Cd^{1/3}$ in Lemma~\ref{ball}.

\begin{theorem}\label{torsion}
Let Assumption~\A\ hold with the stronger estimate~\eqref{free}. Then
\[
  1 \le q(H) \le \frac d8 + c\sqrt{d} + 1,
\]
where $c := \frac14 \smash{\sqrt{5(1+\frac14\ln2)}} < 0.61$.
\end{theorem}

\begin{remark}
(a) Clearly, the upper bound $C_d := \frac d8 + c\sqrt{d} + 1$ in Theorem~\ref{torsion}
is not optimal. But, as discussed above,
$\frac d8$ is the correct leading order term in high dimensions.
It turns out that the bound $C_d$ is also quite good in low dimensions.
Here is a table of approximate values for $q_d := q(-\Delta_{B_d})$ and $C_d\+$:
\begin{center}
\begin{tabular}{|c|c|c|c|c|c|}
\hline
\rule{0pt}{2.3ex}$d$ & 1 & 2 & 3 & 4 & 5 \\ \hline
\rule{0pt}{2.3ex}$q_d$ & 1.2337 & 1.4458 & 1.6449 & 1.8352 & 2.0191 \\ \hline
\rule{0pt}{2.3ex}$C_d$ & 1.7305 & 2.1063 & 2.4238 & 2.7110 & 2.9790 \\ \hline
\end{tabular}
\end{center}
It appears that $C_d$ is never off by more than a factor of $1.5$.

\pagebreak[1]

(b) In \cite[Thm\.1]{beca09} it is shown that
\[
  q(-\Delta_D) \le 3\ln2 \cdot d + 4
\]
for all open subsets $D\subset\R^d$ with $E_0(-\Delta_D)>0$.
Theorem~\ref{torsion} improves the constant $3\ln2$ in front of $d$ to the optimal
$\frac18$ (cf.\ Lemma~\ref{ball}).
\end{remark}

The proofs of Lemma~\ref{ball} and of Theorems~\ref{main-sg} and~\ref{torsion}
will be given in Section~\ref{sec_proofs}. In Section~\ref{sec_weighted}
we provide the necessary tools and prove an auxiliary result.

\section{The method of weighted estimates}\label{sec_weighted}

The aim of this section is to prove the following theorem,
which will be used in the proofs of our two main results.

\begin{theorem}\label{linfty-sg}
Let Assumption~\A\ hold. Then for all $\eps\in(0,1]$ one has
\[
  \|e^{-tH}\|_{\infty\to\infty}
  \le 2^{1/4}M \left( \frac{1+1/\sqrt\eps\,}{2} \right)^{\!d/2} e^{\eps(E_0+\omega)t-E_0t}
  \qquad (t\ge0).
\]
\end{theorem}

\begin{remark}
The precise form of the factor $\bigl( \frac{1+1/\sqrt\eps\,}{2} \bigr)^{\mkern-1mu d/2}$
will be important for the proof of Theorem~\ref{torsion}, where $\omega=0$:
given $t>0$, it is not difficult to show that there exists $\eps\in(0,1]$ with
\[
  \left( \frac{1+1/\sqrt\eps\,}{2} \right)^{\!\smash{d/2}} e^{\eps E_0t-E_0t} < 1
\]
if and only if $E_0t > \frac d8$.
This is the origin of the leading term $\frac d8$ in the upper estimate of
Theorem~\ref{torsion}.
\end{remark}

The proof of Theorem~\ref{linfty-sg} is based on the method of weighted estimates.
We need two ingredients: a result on complex interpolation and a way to derive
$L_\infty{\to}L_\infty$-estimates from weighted $L_2{\to}L_\infty$-estimates.
Our first ingredient is a refinement of Proposition~3.1 from \cite{vog15},
where the case $\eps=\omega=0$ is proved. We denote $\C_+ := \sset{z\in\C}{\Re z>0}$.

\begin{proposition}\label{weighted-est-vog15}
Let $(\Omega,\mu)$ be a measure space, and let $\rho \from \Omega\to\R$ be measurable. Let $E_0\in\R$,
and let $T \from \C_+\to\Ls(L_2(\mu))$ be analytic, $\|T(z)\|_{2\to2} \le e^{-E_0\Re z}$ for all
$z\in\C_+$. Assume that for every $\eps>0$ there exist $C\ge0$ and $\omega\in\R$ such that
\[
  \|e^{-\alpha\rho}\+T(t)\+e^{\alpha\rho}\|_{2\to2} \le Ce^{(1+\eps)\alpha^2t+\omega t} \qquad (\alpha,t>0).
\]
Then $\|e^{-\alpha\rho}\+T(t)\+e^{\alpha\rho}\|_{2\to2} \le e^{\alpha^2t-E_0 t}$ for all $\alpha,t>0$.
\end{proposition}

Here and in the following we denote
\[
  \|e^{-\rho}\lB e^\rho\|_{p\to q} := \sup\set{\|e^{-\rho}\lB e^\rho\- f\|_q}
  {f\in L_p(\mu),\ \|f\|_p\le1,\ e^\rho\- f\in L_2(\mu)}
\]
for a given operator $B \in \Ls(L_2(\mu))$, a measurable function $\rho \from \Omega\to\C$
and $p,q\in[1,\infty]$.

\begin{proof}[\Proof of Proposition~\ref{weighted-est-vog15}]
\hskip 0em minus 0.05em
We define a rescaled analytic function $\tT \from \mkern-1mu\C_+\mkern-1mu\to\Ls(L_2(\mu))$ by
\[
  \tT(z) := \exp\bigl( -\tfrac{\omega z}{1+\eps} \bigr) T\bigl( \tfrac{z}{1+\eps} \bigr).
\]\smallbds
Then
\[
  \|\tT(z)\|_{2\to2}
  \le \exp\bigl( -\tfrac{\omega\Re z}{1+\eps} - E_0 \tfrac{\Re z}{1+\eps} \bigr)
    =  \exp\bigl( -\tfrac{\omega+E_0}{1+\eps} \Re z \bigr)
\]
for all $z\in\C_+$. Moreover,
\[
  \|e^{-\alpha\rho} \+\tT(t)\+ e^{\alpha\rho}\|_{2\to2}
  \le \exp\bigl( -\tfrac{\omega t}{1+\eps} \bigr) \cdot
      C\exp\bigl( {(1+\eps)\alpha^2\tfrac{t}{1+\eps}+\omega\tfrac{t}{1+\eps}} \bigr)
    = C\exp(\alpha^2t)
\]
for all $t>0$. Thus we can apply \cite[Prop\.3.1]{vog15} to obtain
\[
  \|e^{-\alpha\rho} \+\tT(t)\+ e^{\alpha\rho}\|_{2\to2}
  \le \exp\bigl( \alpha^2t - \tfrac{\omega+E_0}{1+\eps}t \bigr)
\]
for all $t>0$ and hence
\[
  e^{-\omega t} \|e^{-\alpha\rho} \+T(t)\+ e^{\alpha\rho}\|_{2\to2}
  = \|e^{-\alpha\rho} \+\tT((1+\eps)t)\+ e^{\alpha\rho}\|_{2\to2}
  \le e^{\alpha^2(1+\eps)t-(\omega+E_0)t}.
\]
Multiplying by $e^{\omega t}$ and letting $\eps\to0$ we obtain the asserted estimate.
\end{proof}

We will work with the weight functions $e^{\pm\alpha\rho_w}$, where $\alpha>0$ and
$\rho_w \from \R^d \to [0,\infty]$ is defined by
\[
  \rho_w(x) := |x-w|,
\]
for given $w\in\R^d$.
Note that $e^{-\alpha\rho_w} \in L_1\cap L_\infty$.
We will need a good estimate for the integral of $e^{-\alpha\rho_w}$.

\begin{lemma}\label{int_est}
\tup{a} For $x>0$ one has $\Gamma(x+\frac12) \le (\rfrac xe)^x \sqrt{2\pi}$.

\tup{b} For $\alpha>0$ one has $\int_{\R^d} e^{-\alpha|y|}\,dy
\le \sqrt2\bigl(\rfrac{2\pi d}{e}\bigr)^{d/2} \alpha^{-d}$.
\end{lemma}
\begin{proof}
(a) We have to show that
\[
  f(x) := x\ln x - x + \ln\mkern-1mu\sqrt{2\pi} - \ln\Gamma(x+\tfrac12) \ge 0
\]
for all $x>0$. More strongly, we show that
\[
  f(x)
  = \int_0^\infty \frac1t \left( \frac1t - \frac{1}{2\sinh(t/2)} \right) e^{-tx}\,dt
  =: g(x)
\]
for all $x>0$, which even implies that $f$ is completely monotone.

First observe that the function $t \mapsto \frac1t \bigl( \frac1t - \frac{1}{2\sinh(t/2)}
\bigr)$ is bounded on $(0,\infty)$, so $g$ is defined and $g(x)\to0$ as $x\to\infty$.
Moreover, by Stirling's formula we obtain
\[
  \lim_{x\to\infty} f(x)
  = \lim_{x\to\infty} \bigl( x\ln x - x - x\ln(x+\tfrac12) + (x+\tfrac12) \bigr)
  = \lim_{x\to\infty} \bigl( x\ln\tfrac{x}{x+1/2} + \tfrac12 \bigr) = 0.
\]
According to \cite[6.4.1]{abst72} we have
$(\ln\Gamma)''(x) = \int_0^\infty \frac{t}{1-e^{-t}} e^{-xt}\,dt$ for all $x>0$.
It follows that
\begin{align*}
f''(x)
 &= \smash[t]{ \frac1x - (\ln\Gamma)''(x+\tfrac12)
  = \int_0^\infty \left( 1 - \frac{te^{-t/2}}{1-e^{-t}} \right) e^{-xt}\,dt } \\
 &= \int_0^\infty \frac1t \left( \frac1t - \frac{1}{e^{t/2}-e^{-t/2}} \right)
          t^2 e^{-xt}\,dt = g''(x)
\end{align*}
for all $x>0$. Together with $\lim_{x\to\infty}(f-g)(x) = 0$ we conclude that $f-g=0$.

(b) Assume without loss of generality that $\alpha=1$.
With $\sigma_{d-1}$ denoting the surface measure of the unit sphere we compute
\[
  \int_{\R^d} e^{-|y|}\,dy = \sigma_{d-1} \int_0^\infty r^{d-1}e^{-r}\,dr
  = \frac{2\pi^{d/2}}{\Gamma(d/2)} \Gamma(d).
\]
By \cite[6.1.18]{abst72} and part (a) we obtain
\[
  \Gamma(2x)/\Gamma(x) = (2\pi)^{-1/2} 2^{2x-1/2} \Gamma(x+\tfrac12) \le 2^{2x-1/2} (\rfrac xe)^x
\]
for all $x>0$. It follows that
\[
  \int_{\R^d} e^{-|y|}\,dy \le 2\pi^{d/2} \cdot 2^{d-1/2} \bigl(\tfrac{d}{2e}\bigr)^{d/2}
  = \sqrt2\bigl(\rfrac{2\pi d}{e}\bigr)^{d/2}. \qedhere
\]
\end{proof}

Based on Lemma~\ref{int_est} we can prove our second ingredient in the method of weighted estimates.

\begin{proposition}\label{weighted-est}
Let $\Omega\subseteq\R^d$ be measurable, $\alpha>0$,
and let $B$ be a bounded operator on $L_2(\Omega)$ satisfying
\[
  \|e^{-\alpha\rho_w} \lB e^{\alpha\rho_w}\|_{2\to\infty} \le 1
\]\smallbds
for all $w\in\Omega$. Then
\[
  \|B\|_{\infty\to\infty} \le 2^{1/4} \bigl(\tfrac{\pi d}{2e}\bigr)^{d/4} \alpha^{-d/2}.
\]
\end{proposition}
\begin{proof}
Let $f\in L_2(\Omega)\cap L_\infty(\Omega)$ have bounded support.
Observe that $\|Bf\|_\infty = \sup_{w\in\Omega} \|e^{-\alpha\rho_w\-} B f\|_\infty$.
For all $w\in\Omega$ we can use the assumption to estimate
\[
  \|e^{-\alpha\rho_w\-} B f\|_\infty
  \le \|e^{-\alpha\rho_w}\- f\|_2
  \le \|e^{-\alpha\rho_w} \|_2 \|f\|_\infty\,.
\]
By Lemma~\ref{int_est}(b) we have
\[
  \|e^{-\alpha\rho_w} \|_2^2 = \smash[t]{\int_\Omega e^{-2\alpha|y|}\,dy}
  \le \sqrt2\bigl(\tfrac{\pi d}{2e}\bigr)^{d/2} \alpha^{-d},
\]
so it follows that $\|Bf\|_\infty
\le 2^{1/4} \smash{\bigl(\frac{\pi d}{2e}\bigr)^{d/4}} \alpha^{-d/2} \|f\|_\infty$.
A simple approximation shows that the same estimate holds for arbitrary
$f\in L_2(\Omega)\cap L_\infty(\Omega)$, which proves the assertion.
\end{proof}

In order to apply Proposition~\ref{weighted-est} in the proof of
Theorem~\ref{linfty-sg}, we need the following weighted estimates
of the free heat semigroup $(e^{t\Delta})_{t\ge0}$ on $\R^d$.

\begin{lemma}\label{free-heat}
Let $w\in\R^d$ and $\alpha>0$. Then
\[
  \|e^{-\alpha\rho_w} e^{t\Delta}\+ e^{\alpha\rho_w}\|_{2\to2} \le e^{\alpha^2t}, \quad
  \|e^{-\alpha\rho_w} e^{t\Delta}\+ e^{\alpha\rho_w}\|_{2\to\infty}
  \le \smash{\bigl(1+\tfrac1\beta\bigr)^{d/4}} (8\pi t)^{-d/4} e^{(1+\beta)\alpha^2t}
\]
for all $t,\beta>0$.
\end{lemma}
\begin{proof}
Let $x,y,w\in\R^d$ and $\alpha,\beta,t>0$.
Let $k_t$ be the convolution kernel of $e^{t\Delta}$.
First observe that
\[
  -\alpha|x-w| + \alpha|y-w| \le \alpha|x-y|
  \le (1+\beta)t\alpha^2 + \tfrac{|x-y|^2}{4(1+\beta)t}\,.
\]
It follows that
\begin{align*}
e^{-\alpha|x-w|} k_t(x-y) e^{\alpha|y-w|}
 &= (4\pi t)^{-d/2} \exp\bigl(-\alpha|x-w|+\alpha|y-w|-\tfrac{|x-y|^2}{4t}\bigr) \\
 &\le (4\pi t)^{-d/2} \exp\bigl(
      (1+\beta)\alpha^2t - \tfrac{\beta}{1+\beta} \tfrac{|x-y|^2}{4t} \bigr) \\
 &= \smash[b]{ \bigl( \tfrac{1+\beta}{\beta} \bigr)^{d/2} }
    e^{(1+\beta)\alpha^2t} k_s(x-y),
\end{align*}\smallbds
with $s = \frac{1+\beta}{\beta}t$. Therefore,
\[
  \|e^{-\alpha\rho_w} e^{t\Delta}\+ e^{\alpha\rho_w}\|_{2\to2}
  \le \smash[t]{ \bigl( \tfrac{1+\beta}{\beta} \bigr)^{d/2} }
      e^{(1+\beta)\alpha^2t} \|e^{s\Delta}\|_{2\to2}\,.
\]
Since $\|e^{z\Delta}\|_{2\to2} \le 1$ for all $z\in\C_+$, the function
$z\mapsto e^{z\Delta}$ satisfies the assumptions of Proposition~\ref{weighted-est-vog15}
with $E_0=0$, and the first assertion follows.

Similarly,
\[
  \|e^{-\alpha\rho_w} e^{t\Delta}\+ e^{\alpha\rho_w}\|_{2\to\infty}
  \le \smash{ \bigl( \tfrac{1+\beta}{\beta} \bigr)^{d/2} }
      e^{(1+\beta)\alpha^2t} \|e^{s\Delta}\|_{2\to\infty}\,.
\]
This implies the second assertion since $\|e^{s\Delta}\|_{2\to\infty} = (8\pi s)^{-d/4}
= \bigl( \tfrac{1+\beta}{\beta} \bigr)^{\mkern-2mu-d/4} (8\pi t)^{-d/4}$.
\end{proof}

Now we are ready to prove the main result of this section.

\begin{proof}[\Proof of Theorem~\ref{linfty-sg}]
We first show the assertion in the case where $E_0=0$ and $a=4$;
then the general assertion is proved by rescaling.

Assumption~\A\ implies that $\|e^{-zH}\|_{2\to2} \le e^{-E_0\Re z} = 1$ for all $z\in\C_+$
and
\[\,
  \|e^{-\alpha\rho_w} e^{-tH} e^{\alpha\rho_w}\|_{2\to2}
  \le Me^{\omega t} \|e^{-\alpha\rho_w} e^{t\Delta} e^{\alpha\rho_w}\|_{2\to2}
  \le Me^{\omega t+\alpha^2t}
  \qquad (w\in\R^d,\ \alpha,t>0),\,
\]
where the last estimate is due to Lemma~\ref{free-heat}.
By Proposition~\ref{weighted-est-vog15} it follows that
\begin{equation}\label{2-2-est}
  \|e^{-\alpha\rho_w} e^{-tH} e^{\alpha\rho_w}\|_{2\to2} \le e^{\alpha^2t}
  \qquad \bigl( w\in\R^d,\ \alpha,t>0 \bigr).
\end{equation}
Applying~\eqref{gub2} and Lemma~\ref{free-heat} again we obtain
\begin{align*}
\|e^{-\alpha\rho_w} e^{-tH} e^{\alpha\rho_w}\|_{2\to\infty}
 &\le Me^{\omega t} \|e^{-\alpha\rho_w} e^{t\Delta} e^{\alpha\rho_w}\|_{2\to\infty} \\
 &\le M(8\pi t)^{-d/4} \bigl(1+\tfrac1\beta\bigr)^{d/4} e^{\omega t+(1+\beta)\alpha^2t}.
\end{align*}
Using the semigroup property and~\eqref{2-2-est}, we deduce for $\eps\in(0,1]$ that
\begin{align*}
  \|e^{-\alpha\rho_w} e^{-tH} e^{\alpha\rho_w}\|_{2\to\infty}
 &\le \|e^{-\alpha\rho_w} e^{-\eps tH} e^{\alpha\rho_w}\|_{2\to\infty} \|e^{-\alpha\rho_w} e^{-(1-\eps)tH} e^{\alpha\rho_w}\|_{2\to2} \\
 &\le M(8\pi\eps t)^{-d/4} \bigl(1+\tfrac1\beta\bigr)^{d/4} e^{\omega\eps t+(1+\beta)\alpha^2\eps t+\alpha^2(1-\eps)t}.
\end{align*}
By Proposition~\ref{weighted-est} it follows that
\begin{align*}
\|e^{-tH}\|_{\infty\to\infty}
 &\le 2^{1/4} \bigl( \tfrac{\pi d}{2e} \bigr)^{d/4} \alpha^{-d/2}
      \cdot M(8\pi\eps t)^{-d/4} \bigl(1+\tfrac1\beta\bigr)^{d/4}
      e^{\omega\eps t+(1+\beta\eps)\alpha^2t} \\
 &=   2^{1/4}M \bigl( \tfrac{d}{16e\eps} \bigr)^{d/4}
      \cdot (\alpha^2t)^{-d/4} \bigl(1+\tfrac1\beta\bigr)^{d/4}
      e^{\omega\eps t+(1+\beta\eps)\alpha^2t}.
\end{align*}
The right hand side in the previous inequality becomes minimal for
$\alpha^2 = \smash{\frac{d/4}{(1+\beta\eps)t}}$.
Then $(\alpha^2t)^{-d/4} = \bigl( \frac4d(1+\beta\eps) \bigr)^{d/4}$
and $(\rfrac1e)^{d/4} e^{(1+\beta\eps)\alpha^2t} = 1$, so
\[
  \|e^{-tH}\|_{\infty\to\infty}
  \le 2^{1/4}M \left( \frac{1+\beta\eps}{4\eps} \bigl(1+\tfrac1\beta\bigr) \right)^{\!d/4}
      e^{\omega\eps t}.
\]
Now the right hand side becomes minimal for $\beta = \eps^{-1/2}$.
Then $\smash{\frac{1+\beta\eps}{4\eps} \bigl(1+\rfrac1\beta\bigr)}
= \frac{(1+\sqrt\eps)^2}{4\eps}$, and we conclude that \vspace{-0.8ex}
\[
  \|e^{-tH}\|_{\infty\to\infty}
  \le 2^{1/4}M \left( \frac{1+1/\sqrt\eps\,}{2} \right)^{\!d/2} e^{\omega\eps t}.
\]

Now we prove the assertion for general $E_0\in\R$ and $a>0$.
Observe that the operator $\tH := \rfrac4a(H-E_0)$ satisfies Assumption~\A\
with $\widetilde E_0 = 0$,\, $\tomega = \rfrac4a(E_0+\omega)$ and $\tilde a = 4$.
Thus we can apply the above to obtain
\[
  \|e^{-\frac a4 t\tH}\|_{\infty\to\infty}
  \le 2^{1/4}M \left( \frac{1+1/\sqrt\eps\,}{2} \right)^{\!d/2} e^{\tomega\eps \frac a4 t}
\]
for all $t\ge0$. The assertion in the general case follows since
$\|e^{-\frac a4 t\tH}\|_{\infty\to\infty} = e^{E_0t} \|e^{-tH}\|_{\infty\to\infty}$
and $\tomega\eps \frac a4 t = \eps(E_0+\omega)t$.
\end{proof}

We conclude this section by explaining how sharp the method of weighted estimates
from Proposition~\ref{weighted-est} is in certain cases.

\begin{remark}
Fix $t>0$ and consider $B=e^{t\Delta}$. Applying Lemma~\ref{free-heat} with $\beta=1$
gives
\[
  \|e^{-\alpha\rho_w} e^{t\Delta} e^{\alpha\rho_w}\|_{2\to\infty}
  \le (4\pi t)^{-d/4} e^{2\alpha^2t}
\]
for all $\alpha>0$ and $w\in\R^d$.
We only need this estimate for $\alpha = \smash{\bigl( \frac{d}{8t} \bigr)^{1/2}}$.
Then $\alpha^2t = \frac d8$, and by Proposition~\ref{weighted-est} we obtain
\[
  \|e^{t\Delta}\|_{\infty\to\infty}
  \le 2^{1/4} \bigl(\tfrac{\pi d}{2e}\bigr)^{d/4} \alpha^{-d/2}
      \cdot (4\pi t)^{-d/4} e^{2\alpha^2t}
  = 2^{1/4} \bigl(\tfrac{d}{8e\alpha^2t}\bigr)^{d/4} e^{2\alpha^2t} = 2^{1/4},
\]
which is quite sharp since $\|e^{t\Delta}\|_{\infty\to\infty} = 1$.
It also follows that in the estimate of Lemma~\ref{free-heat} not much is lost.
\end{remark}

\section{Proof of Theorems~\ref{main-sg} and~\ref{torsion}}\label{sec_proofs}

We first prove bounds for $q(-\Delta_{B_d})$, where $B_d$ is the unit ball in~$\R^d$.

\begin{proof}[\Proof of Lemma~\ref{ball}]
It is easy to show that $(-\Delta_{B_d})^{-1}\ind(x) = \frac1{2d}(1-x^2)$
for all $x\in B_d$, so $\|\Delta_{B_d}^{-1}\ind\|_{\infty} = \frac1{2d}$.
In \cite[Example~5.8]{fmpp07} the estimate $E_0(-\Delta_{B_d}) \ge \frac14 d^2$ is shown.
Now the lower estimate follows from~\eqref{q-formula}:
\[
  q(-\Delta_{B_d}) = E_0(-\Delta_{B_d})\cdot\frac1{2d} \ge \frac d8\,.
\]

It is well-known that $E_0(-\Delta_{B_d}) = j_{(d/2)-1,1}^2$,
where $j_{\nu,1}$ denotes the first positive zero of the Bessel function $J_\nu$.
By \cite{tri49} it follows that $E_0(-\Delta_{B_d}) = \frac14 d^2 + O(d^{4/3})$.
Thus, $q(-\Delta_B) = \frac d8 + O(d^{1/3})$, which implies the upper estimate.
\end{proof}

Theorem~\ref{main-sg} follows from Theorem~\ref{linfty-sg} by an optimization
with respect to $\eps$.

\begin{proof}[\Proof of Theorem~\ref{main-sg}]
Let $t>0$. By Theorem~\ref{linfty-sg} we know that
\[
  \|e^{-tH}\|_{\infty\to\infty}
  \le 2^{1/4}M \left( \frac{1+1/\sqrt\eps\,}{2} \right)^{\!d/2} e^{\eps(E_0+\omega)t-E_0t}
\]
for all $\eps\in(0,1]$. Thus it remains to show that there exists $\eps\in(0,1]$ such that
\begin{equation}\label{aim}
\left( \frac{1+1/\sqrt\eps\,}{2} \right)^{\!d/2} e^{\eps(E_0+\omega)t}
\le \bigl(1+\tfrac{5.56}{d}(E_0+\omega)t\bigr)^{d/4}.
\end{equation}
Setting $x = (E_0+\omega)t/d$ and raising both sides to the power $\frac4d$,
we see that \eqref{aim} is equivalent to
\[
  \left( \frac{1+1/\sqrt\eps\,}{2} \right)^{\!2} e^{4\eps x} \le 1+5.56\+x.
\]

\emph{Case 1:} $x \le \alpha := 0.14$. Then we choose $\eps=1$ and use the inequality
\[
  e^{4x} \le 1 + \rfrac{e^{4\alpha}-1}{\alpha} x \qquad (0\le x\le\alpha),
\]
which is valid due to the convexity of $x \mapsto e^{4x}$.
Now the assertion follows since $\rfrac{e^{4\alpha}-1}{\alpha} < 5.4$.

\emph{Case 2:} $x > \alpha$. Set $\tau = 4e^{-4\alpha}-1 \,(\,>0)$
and choose $\eps=\rfrac{\alpha}{x} \,(\,<1)$. Then
\begin{align*}
\left( \frac{1+1/\sqrt\eps\,}{2} \right)^{\!2} e^{4\eps x}
 &= \left( 1 + \frac{2}{\sqrt\eps} + \frac1\eps \right) \frac{1}{4e^{-4\alpha}}
\\[0.3ex plus 0.2ex]
 &\le \bigl( 1 + \tau + \rfrac{1}{\tau\eps} + \rfrac1\eps \bigr) \tfrac{1}{\tau+1}
    = 1 + \rfrac{1}{\tau\eps} = 1 + \rfrac{1}{\tau\alpha}x,
\end{align*}
and the assertion follows since $\rfrac{1}{\tau\alpha} < 5.56$.
\end{proof}

The proof of Theorem~\ref{torsion} is also based on an optimization with respect to $\eps$;
however, the required estimations are more involved.

\begin{proof}[\Proof of Theorem~\ref{torsion}]
Recall that the lower bound $q(H) \ge 1$ holds by duality and interpolation.

Let $\eps\in(0,1)$, and choose $t_0>0$ such that
\begin{equation}\label{t0-def}
  2^{1/4} \left( \frac{1+1/\sqrt\eps\,}{2} \right)^{\!d/2} e^{-(1-\eps)E_0t_0} = 1.
\end{equation}
The assumption on $H$ implies
$\|e^{-tH}\|_{\infty\to\infty} \le  \|e^{t\Delta}\|_{\infty\to\infty} = 1$
for all $t\ge0$.
Then by the resolvent formula, Theorem~\ref{linfty-sg} and~\eqref{t0-def} we obtain
\begin{align*}
\|H^{-1}\|_{\infty\to\infty}
 &\le \int_0^\infty \|e^{-tH}\|_{\infty\to\infty}\,dt \\
 &\le t_0 + \int_{t_0}^\infty
      2^{1/4} \left( \frac{1+1/\sqrt\eps}{2} \right)^{\!d/2} e^{-(1-\eps)E_0t} \, dt
 = t_0 + \frac1{(1-\eps)E_0}\,.
\end{align*}\smallbds
Thus,
\[
  q(H) = E_0\|H^{-1}\|_{\infty\to\infty} \le E_0t_0 + \frac1{(1-\eps)}\,.
\]
To prove the upper bound, we now show that
\begin{equation}\label{torsion-aim}
  (1-\eps)E_0t_0 + 1 \le (1-\eps) \bigl( \tfrac d8 + c\sqrt{d} + 1 \bigr)
  \avoidbreak
\end{equation}
for a suitable choice of $\eps$.

Set $\gamma := \frac85 c\+$; then $\gamma < \frac85 \cdot 0.61 <1$,\,
$c = \frac58 \gamma$ and $1+\frac14\ln2 = \frac{16}{5} c^2 = \frac54 \gamma^2$.
Further set $x := \frac{\gamma}{\sqrt d} \,(\,<1)$ and choose $\eps := \frac1{(1+2x)^2}$.
Then $\eps\in(0,1)$ as required.
By the choice of~$t_0$ in~\eqref{t0-def}, the left hand side in~\eqref{torsion-aim} equals
\[
  \frac14\ln2 + \frac{d}{2} \ln \frac{1+1/\sqrt\eps}{2} + 1
  = \frac54 \gamma^2 + \frac{d}{2} \ln \frac{1+(1+2x)}{2}
  = \frac54 d\mkern1.5mu x^2 + \frac{d}{2} \ln(1+x),
\]
whereas the right hand side equals
\[
  \left( 1 - \frac1{(1+2x)^2} \right) \cdot
  \frac d8 \- \left( 1 + \frac{8c}{\sqrt d} + \frac8d \+ \right)
  = \frac{x+x^2}{(1+2x)^2} \cdot \frac d2 \bigl( 1 + 5x + \tfrac{8}{\gamma^2} x^2 \bigr),
\]
so that it remains to prove the inequality
\[
  \frac54 d\mkern1.5mu x^2 + \frac{d}{2} \ln(1+x)
  \le \frac{x+x^2}{(1+2x)^2} \cdot \frac d2 \bigl( 1 + 5x + \tfrac{8}{\gamma^2} x^2 \bigr).
\]

Since $\frac{8}{\gamma^2} = 10/(1+\frac14\ln 2) \approx 8.523 > 8.5$, it suffices to show
\[
  \frac52 x^2 + \ln(1+x)
  \le \frac{x+x^2}{(1+2x)^2} \cdot \bigl( 1 + 5x + 8.5\+ x^2 \bigr) =: g(x)
\]
for all $x\in[0,1]$. A straightforward computation yields
\[
  g(x) - \frac52 x^2 = x - \frac{x^2}{2} + \frac{x^3}{2} \cdot \frac{3+x}{(1+2x)^2}
                   \ge x - \frac{x^2}{2} + \frac{x^3}{2} \cdot \frac{4/5}{1+x}\,.
\]
(The last inequality simplifies to $11+4x \ge 11x^2$, which is trivial for $x\in[0,1]$.)
Thus the assertion follows if
\[
  f(x) := x - \frac{x^2}{2} + \frac{x^3}{2} \cdot \frac{4/5}{1+x} - \ln(1+x) \ge 0.
\]
Observe that $f(0)=0$. Now another straightforward computation yields
\[
  f'(x) = \frac{1}{5(1+x)^2} (x^2-x^3) \ge 0
\]
for all $x\in[0,1]$, which completes the proof.
\end{proof}

\subsection*{Acknowledgment}

The author would like to thank Michiel van den Berg for the introduction to the topic
and for helpful discussions.

\end{document}